\documentclass[oneside,english]{amsart}
\usepackage[T1]{fontenc}
\usepackage[latin9]{inputenc}
\usepackage{amsthm}
\usepackage{amstext}
\usepackage{amssymb}
\usepackage{esint}
\usepackage{graphicx}
\usepackage{color}
\usepackage{geometry}
\usepackage{fancyhdr}
\usepackage{xargs}
\usepackage{xifthen}

\makeatletter
\numberwithin{equation}{section}
\numberwithin{figure}{section}
\theoremstyle{plain}
\newtheorem{thm}{\protect\theoremname}[section]
  
  \theoremstyle{plain}
  \theoremstyle{definition}
  
  \theoremstyle{plain}
  
  \newtheorem{cor}[thm]{\protect\corollaryname}
  \theoremstyle{plain}
	\newtheorem{rem}[thm]{\protect\remarkname}
  \theoremstyle{plain}
	\theoremstyle{plain}
	
\makeatother

\usepackage{babel}
  \providecommand{\definitionname}{Definition}
  \providecommand{\lemmaname}{Lemma}
  \providecommand{\theoremname}{Theorem}
  \providecommand{\corollaryname}{Corollary}
  \providecommand{\remarkname}{Remark}
  \providecommand{\propositionname}{Proposition}
  \providecommand{\examplename}{Example}

	\DeclareMathOperator{\loc}{loc}

	\DeclareMathOperator{\cp}{cap}

\newcommandx{\norm}[1][1=\cdot]{\left|{#1}\right|}
\newcommandx{\supnorm}[2][1=\cdot,2=]{\nnorm[#1]_{\infty,#2}}
\newcommandx{\nnorm}[1][1=\cdot]{\left|\left|{#1}\right|\right|}

\newcommandx{\Set}[2][2=]{
    \ifthenelse{\isempty{#2}}
        {\left\{ {#1} \right\}}
        {\left\{  {#1}  \, \middle| \, {#2} \right\}}
}

\newcommand{\R}{\mathbb R}

\newcommand{\Haus}{\mathcal H}

\newcommand{\N}{\mathbb N}

\DeclareMathOperator{\Supp}{Supp}

\theoremstyle{plain}

\theoremstyle{definition}
\newtheorem{dfn}[thm]{Definition}

\theoremstyle{remark}

\begin{document}
\title[On Lipschitz approximations in second order Sobolev spaces]{On Lipschitz approximations in second order Sobolev spaces and the change of variables formula}

\author{Paz Hashash and Alexander Ukhlov}

\begin{abstract}
In this paper we study approximations of functions of Sobolev spaces 
$W^2_{p,\loc}(\Omega)$, $\Omega\subset\mathbb R^n$, by Lipschitz continuous functions. We prove that if $f\in W^2_{p,\loc}(\Omega)$, $1\leq p<\infty$,
then there exists a sequence of closed sets $\{A_k\}_{k=1}^{\infty},A_k\subset A_{k+1}\subset \Omega$, 
such that the restrictions $f \vert_{A_k}$ are Lipschitz continuous functions and $\cp_p\left(S\right)=0$, $S=\Omega\setminus\bigcup_{k=1}^{\infty}A_k$. Using these approximations we prove the change of variables formula in the Lebesgue integral for mappings of Sobolev spaces $W^2_{p,\loc}(\Omega;\mathbb R^n)$ with the Luzin capacity-measure $N$-property.
\end{abstract}
\maketitle

\section{\textbf{Introduction}}
In this paper we study approximations of the second order Sobolev spaces 
$W^2_{p,\loc}(\Omega),$
where 
$\Omega\subset\mathbb R^n$ 
is an open set, by Lipschitz continuous functions. The Lipschitz approximations in Sobolev spaces $W^1_{p,\loc}(\Omega)$ have significant applications in geometric measure theory \cite{H93,VU96} and in the analysis on metric measure spaces \cite{HK98,He01,HKST}. 

In \cite{H93} it was proved that if 
$f\in W^1_{1,\loc}(\Omega)$,
then there exists a sequence of closed sets 
$\{A_k\}_{k=1}^{\infty},A_k\subset A_{k+1}\subset \Omega$, 
for which the restrictions 
$f^* \vert_{A_k}$
are Lipschitz continuous functions on the sets 
$A_k$ 
and $|S|=0$, $S=\Omega\setminus\bigcup_{k=1}^{\infty}A_k$, where we denote by $|S|$ the $n$-dimensional Lebesgue measure of the set $S\subset\R^n$. Here $f^*$ stands for the precise representative of the function $f$ (see 
$\eqref{precise representative}$ below). 

In the case of the second order Sobolev spaces $W^2_{1,\loc}(\Omega)$ we prove the following refined version of this Lipschitz approximation:
{ \it Let $f\in W^2_{1,\loc}(\Omega)$. Then there exists a sequence of closed sets 
$\{A_k\}_{k=1}^{\infty},A_k\subset A_{k+1}\subset \Omega$, such that the restrictions 
$f^*\vert_{A_k}$ are Lipschitz continuous functions $\Haus^{n-1}$-a.e. in $A_k$ and $\Haus^{n-1}\left(S\right)=0$, $S=\Omega\setminus\bigcup_{k=1}^{\infty}A_k$,
where $\Haus^{n-1}$ is the $(n-1)$-Hausdorff measure.}

Using this Lipschitz approximation we obtain the following change of variables formula in the Lebesgue integral for twice weakly differentiable mappings: 
Let $\varphi:\Omega\to\mathbb R^n$ be a mapping which belongs to the Sobolev space
$W^2_{1,\loc}(\Omega;\mathbb R^n)$. Then there exists a Borel set 
$S\subset \Omega$, $\Haus^{n-1}\left(S\right)=0$, 
such that the mapping
$\varphi:\Omega\setminus S \to \mathbb R^n$ 
has the Luzin $N$-property (an image of a set of Lebesgue measure zero has Lebesgue measure zero) and the change of variables formula
\begin{equation}
\label{eq:change of variables formula 1}
\int\limits_A u\circ\varphi (x)|J(x,\varphi)|~dx=\int\limits_{\mathbb R^n\setminus \varphi(S)} u(y)N_\varphi(A,y)~dy,
\end{equation}
where 
$N_\varphi(A,y)$ 
is the multiplicity function which is defined as the number of preimages of $y$ in $A$ under $\varphi$ and $J(x,\varphi)=\det\left(D\varphi(x)\right)$ is the Jacobian of $\varphi$ at the point $x$, holds for every  measurable set 
$A\subset \Omega$ 
and every nonnegative measurable function 
$u: \mathbb R^n\to\mathbb R$.

If the mapping $\varphi$ possesses the Luzin Hausdorff-Lebesgue measure $N$-property (an image of a set of $\Haus^{n-1}$-Hausdorff measure zero has $n$-dimensional Lebesgue measure zero), then 
$|\varphi (S)|=0$ and the integral on the right hand side of $\eqref{eq:change of variables formula 1}$
can be rewritten as an integral on $\R^n$.

In the case of the second order Sobolev spaces $W^2_{p,\loc}(\Omega)$, $1<p<\infty$, we prove the refined version of the Lipschitz approximation in capacitory terms:
{\it Let
$f\in W^2_{p,\loc}(\Omega)$, $1< p<\infty$.
Then there exists a sequence of closed sets 
$\{A_k\}_{k=1}^{\infty},A_k\subset A_{k+1}\subset \Omega$, 
such that the restrictions 
$f^* \vert_{A_k}$ 
are Lipschitz continuous functions $p$-quasieverywhere in $A_k$(means that outside a set of $\cp_p$ zero)} and $\cp_p\left(S\right)=0$, $S=\Omega\setminus\bigcup_{k=1}^{\infty}A_k$,
where $\cp_p$ is the $p$-capacity.

Using this Lipschitz approximation we obtain the following capacitory version of the change of variables formula in the Lebesgue integral for twice weakly differentiable mappings: 
Let $\varphi:\Omega\to\mathbb R^n$ be a mapping which belongs to the Sobolev space
$W^2_{p,\loc}(\Omega;\mathbb R^n)$, $1< p<\infty$. Then there exists a Borel set 
$S\subset \Omega,\cp_{p}(S)=0$, 
such that the mapping
$\varphi:\Omega\setminus S \to \mathbb R^n$ 
has the Luzin $N$-property and the change of variables formula
\begin{equation}
\label{eq:change of variables formula 2}
\int\limits_A u\circ\varphi (x)|J(x,\varphi)|~dx=\int\limits_{\mathbb R^n\setminus \varphi(S)} u(y)N_\varphi(A,y)~dy,
\end{equation}
holds for every  measurable set 
$A\subset \Omega$ 
and every nonnegative measurable function 
$u: \mathbb R^n\to\mathbb R$.

If the mapping $\varphi$ possesses the Luzin capacity-measure $N$-property (an image of a set of capacity zero has Lebesgue measure zero), then $|\varphi (S)|=0$ and the integral on the right hand side of $\eqref{eq:change of variables formula 2}$
can be rewritten as an integral on $\R^n$. 
Since for any 
$E\subset\R^n$ 
we have the inequality 
$|E|\leq \cp_p(E),1<p<\infty$ \cite{K21}, then mappings which possess the Luzin measure $N$-property have the Luzin capacity-measure $N$-property.
 
Note, mappings which possess the Luzin capacity-capacity $N$-property (an image of a set of capacity zero has capacity zero) and so possess the Luzin capacity-measure $N$-property arise, in particular, in the geometric theory of composition operators on Sobolev spaces, see, for example, \cite{GU20,U93,VU04}. The geometric theory of composition operators on Sobolev spaces have applications in the spectral theory of elliptic operators, see, for example, \cite{GU17}.  We note also the generalized quasiconformal mappings, so-called $Q$-mappings, which possess the Luzin capacity-capacity $N$-property. The theory of $Q$-mappings is intensively developed in the last decades. See, for example, \cite{MRSY09}. 

On the base of the refined Lipschitz approximation we obtain the Luzin type theorem for capacity for second order Sobolev spaces. It is known \cite{BH93,evans2015measure,HKM,M} that if  $f\in W^1_{p,\loc}(\Omega)$, then for each $\varepsilon>0$ there exists an open set $U_{\varepsilon}$ of $p$-capacity less than $\varepsilon$ such that $f^*$ is continuous on the set $\Omega\setminus U_{\varepsilon}$.
We prove the following refined version: 
If $f\in W^2_{p,\loc}(\Omega)$, then for each 
$\varepsilon>0$ there exists an open set $U_{\varepsilon}$ of $p$-capacity less than $\varepsilon$ such that $f^*$ is Lipschitz continuous on the set $\Omega\setminus U_{\varepsilon}$.

The refined Luzin theorem was obtained in \cite{Ma93}. It was proved that Sobolev functions coincide with H\"older continuous functions on the complement of a set of arbitrary small capacity. This result plays a crucial role in the refined versions of the change of variables formula for the Sobolev mappings \cite{Ma94,MM95}. 
In the case of higher order Sobolev spaces the Luzin type theorem considered in \cite{Z89}. The refined Luzin theorem on the H\"older continuity of higher order Sobolev spaces was proved in \cite{BHS02}.
In the recent work \cite{BKK15} the Luzin type theorem for functions of $W^{n}_1(\mathbb R^n)$ in the terms of the Hausdorff content was proved.

We prove the following refined version of the Luzin type theorem:
{\it If 
$f\in W_{1,\loc}^2(\Omega)$ 
and 
$B\subset \Omega$ 
is a Borel set such that 
$\Haus^{n-1}(B)<\infty$, 
then for every 
$\varepsilon>0$ 
there exists a Borel set 
$B_{\varepsilon}\subset B$ 
such that 
$f^* \vert_{B_{\varepsilon}}$ 
is Lipschitz continuous $\Haus^{n-1}-$a.e in $B_\epsilon$  and 
$\Haus^{n-1}(B\setminus B_{\varepsilon})<\varepsilon$.}

In the case of the second order Sobolev spaces $W^2_{p,\loc}(\Omega)$, $1<p<\infty$, we prove:
{\it If $f\in W^2_{p,\loc}(\Omega),1<p<\infty$,
and 
$A\subset\Omega$
is a compact set with an additional assumption
(see Corollary $\ref{cor:Luzin property for capacity}$), then for every 
$\varepsilon>0$ 
there exists a Borel set 
$A_{\varepsilon}\subset A$ 
such that 
$f^*\vert_{A_\varepsilon}$ 
is Lipschitz continuous 
$p-$quasieverywhere in $A_\varepsilon$ and 
$\cp_{p}(A\setminus A_{\varepsilon})<\varepsilon$. }
  
The suggested methods are based in the non-linear potential theory \cite{HKM,M} and the Chebyshev type inequality for capacity \cite{evans2015measure,VCh1, VCh2}. This paper is organized as follows:
Section 2 contains definitions and we consider Lipschitz approximations in Sobolev spaces $W^2_{p,\loc}(\Omega)$ and prove Luzin type theorems for Lipschitz $p-$quasicontinuity(see Definition $\ref{def:Lipschitz quasicontinuity}$). In Section 3 we prove the change of variables formula for twice weakly differentiable mappings which have the Luzin capacity-measure $N$-property.

\section{\textbf{Lipschitz approximations of Sobolev functions}}

Let 
$E\subset \mathbb{R}^n$
be a measurable set. Recall that the Lebesgue space 
$L_{p}(E),1\leq p<\infty$, 
is defined as the space of $p$-integrable functions with the norm 
$$ \|f\mid L_p(E)\|=
\left(\int\limits_{E}|f(x)|^p~dx\right)^{1/p}, \,\, 1\leq p<\infty.
$$
By 
$L_{p,\loc}(E)$ 
we denote the space of locally $p$-integrable functions, means that,
$f\in L_{p,\loc}(E)$ 
if and only if  
$f\in L_p(F)$ 
for every compact subset 
$F\subset E.$

Let 
$\Omega\subset \R^n$
be an open set. The Sobolev space 
$W^m_p(\Omega),m\in \N,1\leq p<\infty$,
is defined as the normed
space of functions $f\in L_p(\Omega)$ such that the partial derivatives of order less than or equal to $m$
exist in the weak sense and belong to $L_p(\Omega)$. The space is equipped with the norm
$$
\|f\mid W^m_p(\Omega)\|=
\sum\limits_{|\alpha|\leq m}\left(\int\limits_{\Omega} |D^{\alpha}f(x)|^p\,dx\right)^{\frac{1}{p}}<\infty,
$$   
where 
$D^{\alpha} f$ 
is the weak derivative of order $\alpha$ of the function $f$ which is defined by the following formula:
$$
\int\limits_{\Omega} f D^{\alpha}\eta~dx=(-1)^{|\alpha|}\int\limits_{\Omega} (D^{\alpha}f) \eta~dx, \quad \forall \eta\in C_c^{\infty}(\Omega),
$$
where
$\alpha:=(\alpha_1,\alpha_2,...,\alpha_n)$ 
is a multiindex, 
$\alpha_i=0,1,...,$ 
$|\alpha|=\alpha_1+\alpha_2+...+\alpha_n$.
The space 
$C_c^{\infty}(\Omega)$ 
is the space of smooth functions with compact support in 
$\Omega.$ 
The Sobolev space 
$W^m_{p,\loc}(\Omega)$ 
is defined as follows:
$f\in W^m_{p,\loc}(\Omega)$ 
if and only if 
$f\in W^m_p(U)$ 
for every open and bounded set 
$U\subset  \Omega$ 
such that 
$\overline{U}  \subset \Omega$, where $\overline{U} $ stands for the topological closure of the set $U$.

Let us recall the definition of the capacity \cite{evans2015measure,HKM,M}.
Suppose $\Omega$ is an open set in $\R^n$ and  $F\subset\Omega$ is a compact set. The $p$-capacity of $F$ with respect to $\Omega$ is defined by
\begin{equation}
\cp_p(F;\Omega) =\inf\{\|\nabla f|L_p(\Omega)\|^p\},
\end{equation} 
where the inferior is taken over all 
$f\in C_c^{\infty}(\Omega)$ such that 
$f\geq 1$ on $F$. 
If 
$U\subset\Omega$ 
is an open set, we define
\begin{equation}
\cp_{p}(U;\Omega)=\sup\{\cp_{p}
(F;\Omega)\,:\,F\subset U,\,\, F\,\,\text{is compact}\}.
\end{equation}
In the case of an arbitrary set 
$E\subset\Omega$
we define
\begin{equation}
\label{eq:outer regularity of capacity}
\cp_{p}(E;\Omega)=\inf\{\cp_{p}(U;\Omega)\, :\,\,E\subset U\subset\Omega,\,\, U\,\,\text{is open}\}.
\end{equation}
The $p$-capacity 
is an outer measure on
$\Omega$,
i.e it is a set function
which is defined on every subset of 
$\Omega$
such that
$\cp_p(\emptyset;\Omega)=0$
and it is $\sigma$-subadditive, means that, if 
$E\subset \bigcup_{l\in \N}E_l$,
then 
$\cp_p\left(E;\Omega\right)\leq \sum_{l\in \N}\cp_p\left(E_l;\Omega\right).$
We write for short
$\cp_p(E)=\cp_p(E;\R^n),E\subset\R^n.$

The notion of the $p$-capacity permits us to refine the notion of Sobolev functions \cite{HM72,M}. Let 
$f\in L_{1,\loc}(\Omega).$
The \textbf{precise representative} of 
$f$ is defined by 
\begin{equation}
\label{precise representative}
f^*:\Omega\to \R,\quad f^*(x):=
\begin{cases}
\lim_{\epsilon\downarrow 0}f_{B(x,\epsilon)},\quad & \text{if the limit exists};\\
0, \quad & \text{otherwise};
\end{cases}
\end{equation}
where
\begin{equation}
f_{B(x,\epsilon)}:=\fintop_{B(x,\epsilon)}f(y)dy=\frac{1}{\left|B(x,\epsilon)\right|}\intop_{B(x,\epsilon)}f(y)dy.
\end{equation}    
Recall that a function 
$f$ 
is termed \textbf{p-quasicontinuous} in an open set 
$\Omega$ 
if and only if for every 
$\varepsilon >0$ 
there exists an open  set 
$U_{\varepsilon}\subset \Omega$ 
such that the $p$-capacity of 
$U_{\varepsilon}$ 
is less than 
$\varepsilon$ 
and on the set 
$\Omega\setminus U_{\varepsilon}$ 
the function  $f$ is continuous. 
If $f\in W^1_{p,\loc}(\Omega)$, then $f^*$ is $p-$quasicontinuous and there exists a Borel set $E$ such that $\cp_p(E)=0$ and $f^*(x)=\lim_{\epsilon\to 0^+}f_{B(x,\epsilon)},\forall x\in \Omega\setminus E$  \cite{evans2015measure,HM72,HKM,M}.
In case $f\in W^1_{p,\loc}(\Omega)$ the refined function 
$f^*$ 
is called \textbf{the unique quasicontinuous representation} (or \textbf{the canonical representation}) of  $f$.
  
\subsection{Lipschitz $p$-quasicontinuous functions}

Let us recall the following two theorems from the theory of Sobolev spaces, see for example \cite{evans2015measure,K21,M}: 

\begin{thm}
$($The local Poincar\'e inequality for functions of $W^1_{1,\loc}(\mathbb R^n)$$)$.\\ 
There exists a constant 
$C=C(n)$ 
such that 
\begin{equation}
\fintop_{B(x,r)}|f(y)-f_{B(x,r)}|dy\leq Cr\fintop_{B(x,r)}|\nabla f(y)|dy,
\end{equation}
for every ball $B(x,r)\subset \R^n$ and every 
$f\in W^1_1(B(x,r))$.       
\end{thm} 

\begin{thm}$($The Chebyshev type inequality for capacity$)$.
\\Assume 
$p\geq 1,f\in W^1_p(\R^n)$ 
and let  
$\alpha>0.$ 
Define
\begin{equation*}
R_\alpha:=\Set{x\in \R^n}[\sup_{\epsilon>0}f_{B(x,\epsilon)}\leq \alpha].
\end{equation*}
Then
\begin{equation*}
\cp_{p}(\R^n\setminus R_\alpha)\leq \frac{C(n,p)}{\alpha^p}
\|\nabla f|L_p(\R^n)\|^p.
\end{equation*}         
\end{thm}

The following theorem gives us an approximation of twice weakly differentiable functions by Lipschitz continuous functions.

\begin{thm}
\label{lem:LipCap}
Let 
$\Omega\subset \R^n$
be an open set and 
$f\in W^2_{p,\loc}(\Omega),1\leq p<\infty$.
Then there exists a sequence of closed sets 
$\{C_k\}_{k=1}^{\infty}$
such that for every
$k=1,2,...$,
$C_k\subset C_{k+1}\subset\Omega$,
the restriction 
$f^* \vert_{C_k}$ 
is a Lipschitz continuous function defined $p$-quasieverywhere in
$C_k$
and 
\begin{equation}
\cp_{p}\left(\Omega\setminus\bigcup_{k=1}^{\infty}C_k\right)=0.
\end{equation}
\end{thm}

\begin{proof}
Assume first 
$f\in W^2_p(\R^n).$
Then for any 
$\alpha>0$
we define a set 
\begin{equation}
R_\alpha=\Set{x\in \R^n} [\sup_{r>0}|\nabla f|_{B(x,r)}\leq \alpha].
\end{equation}
The set $R_\alpha$
is closed: fix an arbitrary number 
$\varepsilon>0$
and let 
$\{x_i\}_{i=1}^\infty \subset R_\alpha$
be a sequence of points such that $x_i\to x$
as 
$i\to \infty.$ 
Since for every $r>0$ the function $z\longmapsto|\nabla f|_{B(z,r)}$
is continuous, there exists a number $r_0>0$ 
such that 
$$
\left||\nabla f|_{B(z,r)}-|\nabla f|_{B(x,r)}\right|\leq \varepsilon,\quad \forall z\in B(x,r_0).
$$
Thus, for sufficiency large numbers $i$ such that $x_i\in B(x,r_0)$ we have 
$$
|\nabla f|_{B(x,r)}\leq \varepsilon+|\nabla f|_{B(x_i,r)}\leq \varepsilon+\alpha.
$$
Hence $\sup_{r>0}|\nabla f|_{B(x,r)}\leq \varepsilon+\alpha$ and since $\varepsilon>0$ is arbitrary we get $x\in R_\alpha$ and so $R_\alpha$ is a closed set. 
Note, in addition, that by the definition of the sets $R_{\alpha}$, $\alpha>0$,
we have 
$R_{\alpha_1}\subset R_{\alpha_2}$ for numbers 
$\alpha_2\geq\alpha_1>0$.
Now, let 
$x\in R_\alpha$. Since $|\nabla f|\in L_p(\mathbb R^n)$, then by the local Poincar\'e inequality we have
\begin{equation}
\fintop_{B(x,r)}|f(y)-f_{B(x,r)}|dy\leq C(n,p)r\fintop_{B(x,r)}|\nabla f(y)|dy\leq Cr\alpha. 
\end{equation}
Hence
\begin{multline}
|f_{B(x,r/2^{k+1})}-f_{B(x,r/2^k)}|
\leq \fintop_{B(x,r/2^{k+1})}|f(y)-f_{B(x,r/2^k)}|dy\\
\leq 2^n\fintop_{B(x,r/2^k)}|f(y)-f_{B(x,r/2^k)}|dy
\leq \frac{Cr\alpha}{2^k},
\end{multline}
where
$C$
is a constant which is dependent on $n$ and $p$ only.   
Since 
$f\in W^2_{p}(\R^n)$, then $f\in W^1_{p}(\R^n)$ and \cite{HM72,M} there exists a Borel set $E\subset \R^n$ such that $\cp_{p}(E)=0$ and 
\begin{equation}
\label{eq:Cap_p(E)=0 and there is a convergence of the averages of f to f^*}
\lim_{r\to 0^+}f_{B(x,r)}=f^*(x) \quad \text{for all} \quad x\in \R^n\setminus E.
\end{equation}
Therefore, for every $x\in R_\alpha\setminus E$ it follows that 
\begin{equation}
\label{eq:estimate for |f^*(x)-f_B(x,r)|}
|f^*(x)-f_{B(x,r)}|=\left|\sum_{k=0}^\infty \left[f_{B(x,r/2^{k+1})}-f_{B(x,r/2^k)}\right]\right|\leq Cr\alpha. 
\end{equation}
Now, we take arbitrary points $x,y\in R_\alpha$ such that $x\neq y$ and we set $r=|x-y|$. Then
\begin{multline}
\label{eq:estimate for the averages (f)_x,r and (f)_y,r}
|f_{B(x,r)}-f_{B(y,r)}|\leq \fintop_{B(x,r)\cap B(y,r)} \left(|f_{B(x,r)}-f(z)|+|f(z)-f_{B(y,r)}|\right)dz\\
\leq 2^n\left(\fintop_{B(x,r)}|f_{B(x,r)}-f(z)|dz+
\fintop_{B(y,r)}|f(z)-f_{B(y,r)}|dz \right)
\leq Cr\alpha.
\end{multline}
By the results 
$\eqref{eq:estimate for |f^*(x)-f_B(x,r)|}$ 
and 
$\eqref{eq:estimate for the averages (f)_x,r and (f)_y,r}$ 
and the triangle inequality we obtain 
\begin{equation}
\label{eq:f^* is Lipschitz}
|f^*(x)-f^*(y)|\leq C\alpha|x-y|, \quad \text{for any}\quad x,y\in R_\alpha \setminus E. 
\end{equation}
In particular, for every integer
$k\geq1$
the function
$f^*$
is Lipschitz continuous on 
$R_k\setminus E,\cp_p(E)=0$.
Now, since $f\in W^2_{p}(\R^n)$, then 
$\nabla f\in W^1_{p}(\R^n;\R^n)$
and by the Chebyshev type inequality for capacity \cite{evans2015measure,K21} we have the following  estimate  
\begin{equation}
\label{eq:estimate for the capacity by the derivative}
\cp_{p}(\R^n\setminus R_\alpha)\leq \frac{C(n,p)}{\alpha^p}
\nnorm[{\nabla}^2 f \mid L_p(\R^n)]^p,\quad \text{for any number}\quad \alpha>0.
\end{equation}
Therefore, by 
$\eqref{eq:estimate for the capacity by the derivative}$
\begin{equation}
\cp_{p} (\R^n\setminus\cup_{k=1}^\infty R_k)=0.
\end{equation}
It completes the proof in the case $f\in W^2_p(\R^n).$

Now, assume 
$f\in W^2_{p,\loc}(\Omega)$.
Let 
$\Omega_k\subset \Omega_{k+1}\subset \Omega$
be a nested sequence for 
$\Omega$,
i.e., 
$\forall k,$ $\Omega_k$
is an open set such that 
$\overline{\Omega}_k\subset \Omega_{k+1}$,
$\overline{\Omega}_k$
is compact and the union of the sets 
$\Omega_k$
equals to 
$\Omega$.

For each $k$ we choose a function 
$\zeta_k\in C^\infty _c(\Omega_{k+1})$ 
such that
$\zeta_k\equiv 1$
on
$\overline{\Omega}_k$.

It follows that 
$f_k:=f\zeta_k\in W_p^2(\Omega_{k+1}),\Supp(f_k)\subset \Omega_{k+1}$,
where 
$\Supp(f_k)$
denotes the support of 
$f_k$.
We extend 
$f_k$
from
$\Omega_{k+1}$
to
$\R^n$
by zero and denote the extension again by
$f_k$.
Then 
$f_k\in W^2_p(\R^n)$
and we get by the previous case a non decreasing sequence of closed sets
$\{A_l^k\}_{l=1}^\infty$
such that
$f_k^*|_{A_l^k}$
is Lipschitz continuous in $A_l^k$ up to a set of $p$-capacity zero and
$\cp_{p} (\R^n\setminus\cup_{l=1}^\infty A_l^k)=0$.
Let us choose a sequence of numbers $\alpha_k\in \N$ 
such that 
\begin{equation}
\alpha_k^p\geq 2^k\nnorm[{\nabla}^2 f_k \mid L_p(\R^n)]^p.
\end{equation}

Using the Chebyshev type inequality for capacity we have

\begin{equation}
\cp_{p}(\R^n\setminus A^k_{\alpha_k})\leq \frac{C(n,p)}{\alpha_k^p}
\nnorm[{\nabla}^2 f_k \mid L_p(\R^n)]^p\leq \frac{C(n,p)}{2^k}.
\end{equation} 

For a sequence 
$\{\alpha_k\}_{k=1}^\infty$ 
as above let us define a sequence of sets

\begin{equation}
\{B_l\}_{l=1}^\infty,\quad B_l=\bigcap_{k=l}^\infty A_{\alpha_k}^k.
\end{equation}
It follows that
\begin{equation}
\R^n\setminus \bigcup_{l=1}^\infty B_l=\R^n\setminus \liminf_{l\to \infty}A_{\alpha_l}^l=\limsup_{l\to \infty}\left(\R^n \setminus A_{\alpha_l}^l\right)=\bigcap_{j=1}^\infty \bigcup_{k=j}^\infty \left(\R^n \setminus A_{\alpha_k}^k\right).
\end{equation}
Therefore, for any $j\geq 1$ we obtain
\begin{equation}
\cp_p\left(\R^n\setminus \bigcup_{l=1}^\infty B_l\right)\leq
\sum_{k=j}^\infty\cp_p\left(\R^n \setminus A_{\alpha_k}^k\right)\leq C(n,p)\sum_{k=j}^\infty \frac{1}{2^k}.
\end{equation}
Thus,
\begin{equation}
\cp_p\left(\R^n\setminus \bigcup_{l=1}^\infty B_l\right)=0.
\end{equation}
Now, define a sequence of sets

\begin{equation}
\{C_l\}_{l=1}^\infty,\quad C_l=B_l\cap \overline{\Omega}_l.
\end{equation}
The sequence $\{C_l\}_{l=1}^\infty$ is monotone increasing since 
$\{B_l\}_{l=1}^\infty,
\{\overline{\Omega}_l\}_{l=1}^\infty$ are monotone increasing. $C_l$ is closed as an intersection of closed sets $B_l,\overline{\Omega}_l$.
Since  
\begin{equation}
\Omega\setminus \bigcup_{l=1}^\infty C_l
=\Omega\setminus \bigcup_{l=1}^\infty B_l\cap \overline{\Omega}_l 
=\Omega\setminus \bigcup_{l=1}^\infty B_l\cap\Omega
=\Omega\cap \left(\R^n\setminus \bigcup_{l=1}^\infty B_l\right)
\subset 
\R^n\setminus \bigcup_{l=1}^\infty B_l,
\end{equation}
then
\begin{equation}
\cp_p\left(\Omega\setminus \bigcup_{l=1}^\infty C_l\right)=0.
\end{equation}
At last, since 
$C_l\subset A^l_{\alpha_l}\cap \overline{\Omega
}_l$ and 
$f_l^*\vert_{A_{\alpha_l}^l\cap\overline{\Omega}_l}=f^*\vert _{A^l_{\alpha_l}\cap\overline{\Omega}_l}$
is Lipschitz continuous in $A^l_{\alpha_l}\cap \overline{\Omega}_l$ up to a set of $p$-capacity zero, then $f^*\vert_{C_l}$ is a Lipschitz continuous function in $C_l$ up to a set of $p$-capacity zero.
It completes the proof of the theorem. 
\end{proof}

We denote by 
${\Haus}^s,s\geq 0$, the $s$-dimensional Hausdorff measure. For Borel sets 
$B\subset \R^n$ 
such that
$\Haus^{n-1}(B)<\infty$
it follows that
\begin{equation}
\label{eq:mutual absolute continuity of cap_1 and Hausdorff measure}
\Haus^{n-1}(B)=0\quad \Longleftrightarrow\quad \cp_{1}(B)=0.
\end{equation} 
For a proof of this relation between the (outer) measures 
$\Haus^{n-1},\cp_{1}$ 
see for example Theorem 3 in Section 5.6 in the book \cite{evans2015measure}.

The following corollary is an immediate consequence of Theorem 
$\ref{lem:LipCap}$
and 
$\eqref{eq:mutual absolute continuity of cap_1 and Hausdorff measure}.$
\begin{cor}
\label{lem:LipHaus}
Let 
$\Omega\subset \R^n$
be an open set and 
$f\in W^2_{1,\loc}(\Omega)$.
Then there exists a sequence of closed sets 
$\{C_k\}_{k=1}^{\infty}$
such that for every
$k=1,2,...$,
$C_k\subset C_{k+1}\subset\Omega$,
the restriction 
$f^* \vert_{C_k}$ 
is a Lipschitz continuous function defined $1$-quasieverywhere in
$C_k$
and 
\begin{equation}
\Haus^{n-1}\left(\Omega\setminus\bigcup_{k=1}^{\infty}C_k\right)=0.
\end{equation}
\end{cor}

Now we define the notion of Lipschitz $p$-quasicontinuous functions.

\begin{dfn}
\label{def:Lipschitz quasicontinuity}
Let 
$\Omega\subset \R^n$ 
be an open set and let 
$f:\Omega\to \R$ 
be a function. The function $f$ is called
\textbf{$p$-Lipschitz quasicontinuous} 
if for each 
$\varepsilon>0$ 
there exists an open set 
$V\subset \Omega$ 
such that
\begin{equation}
\cp_p(V)\leq \varepsilon
\end{equation} 
and
\begin{equation}
f\vert_{\Omega\setminus V\quad} \text{is Lipschitz continuous}.
\end{equation}
\end{dfn}

The following corollary is an immediate consequence of the previous theorem.

\begin{cor}
Let $f\in W^2_p(\R^n),1\leq p<\infty.$ Then $f^*$ is $p-$Lipschitz quasicontinuous. 
\end{cor}

\begin{proof}
Fix 
$\varepsilon>0.$ 
By 
$\eqref{eq:estimate for the capacity by the derivative}$ 
there exists a big enough 
$\alpha>0$ 
such that
$\cp_{p}\left(\R^n\setminus (R_\alpha\setminus E)\right)\leq \frac{\varepsilon}{2}$ 
and 
$f^*$ 
is Lipschitz continuous in 
$R_\alpha\setminus E$, where 
$R_\alpha$ and $E$ 
are the same as in the proof of Theorem 
$\ref{lem:LipCap}$.
By the outer regularity of $p$-capacity
$\eqref{eq:outer regularity of capacity}$
there exists an open set 
$V$ 
such that 
$\R^n\setminus (R_\alpha\setminus E) \subset V$, $\cp_p(V)\leq \varepsilon$ 
and 
$f^*$ 
is Lipschitz continuous in 
$\R^n\setminus V$.
\end{proof}

\subsection{The refined Luzin type theorem for functions of second order Sobolev spaces} 
  
\begin{thm}
\label{th:Luzin's thorem for capacity}
Let 
$\Omega\subset \R^n$ 
be an open set and 
$f\in W_{1,\loc}^2(\Omega).$ 
Let 
$B\subset \Omega$ 
be a Borel set such that 
$\Haus^{n-1}(B)<\infty.$ 
Then for any real number 
$\varepsilon>0$ 
there exists a Borel set 
$B_{\varepsilon}\subset B$ 
such that 
$f^*\vert_{B_{\varepsilon}}$ 
is Lipschitz continuous 
in 
$B_\varepsilon$ up to a set of $1$-capacity zero
and 
$\Haus^{n-1}(B\setminus B_{\varepsilon})<\varepsilon.$  
\end{thm}

\begin{proof}
By Theorem 
$\ref{lem:LipCap}$,
since 
$f\in W_{1,\loc}^2(\Omega)$, 
there exists a  sequence of closed sets 
$B_k\subset B_{k+1}\subset\Omega$ 
such that 
$f^*\vert_{B_k}$ 
is Lipschitz continuous
in 
$B_k$ up to a set of $1$-capacity zero
and
$\cp_{1}\left(\Omega\setminus \bigcup_{k=1}^\infty B_k\right)=0$. Since  
\begin{equation}
\cp_{1}\left(B\setminus \bigcup_{k=1}^\infty B_k\right)=0,
\end{equation}
then
\begin{equation}
\Haus^{n-1}\left(B\setminus \bigcup_{k=1}^\infty B_k\right)=0.
\end{equation}
By the assumption that 
$\Haus^{n-1}(B)<\infty$, we have
 
\begin{equation}
\lim_{l\to \infty}\Haus^{n-1}\left(B\setminus \bigcup_{k=1}^l B_k\right)=0.
\end{equation}

Choose a big enough natural number $l$ such that $\Haus^{n-1}\left(B\setminus \bigcup_{k=1}^l B_k\right)<\varepsilon$ and set $B_{\varepsilon}:=\bigcup_{k=1}^l B_k\cap B.$ It follows that $B_{\varepsilon}\subset B$ is a Borel set for which
$f^*\vert_{B_{\varepsilon}}$ 
is Lipschitz continuous
in $B_\varepsilon$ up to a set of $1$-capacity zero
and 
$\Haus^{n-1}(B\setminus B_{\varepsilon})<\varepsilon.$ 
\end{proof}

\begin{rem}
If we assume in Theorem
$\ref{th:Luzin's thorem for capacity}$
that the set 
$B$
is closed, then we see from the proof of this theorem that we can also choose the set 
$B_\varepsilon$
to be closed. 
\end{rem}

Before we introduce another capacity version of Luzin's theorem we recall that for a monotone decreasing sequence of compact sets 
$C_{k+1}\subset C_k\subset \R^n$ 
we have 
$$
\lim\limits_{k\to \infty}\cp_{p}(C_k)=\cp_{p}\left(\bigcap_{k=1}^\infty C_k\right).
$$
The proof of this formula can be found, for example, in \cite{evans2015measure}. We denote by $Int(B)$ the topological interior of a set $B.$     

\begin{thm}
\label{thm:Luzin's theorem for union of compact sets}
Let 
$\Omega\subset \R^n$ 
be an open set and 
$f:\Omega\to\mathbb R$ 
be a function. Assume the existence of a sequence of closed sets 
$\{B_k\}_{k=1}^{\infty}$,
$B_k\subset B_{k+1}\subset\Omega$
for which the restrictions 
$f \vert_{B_k}$ 
are Lipschitz continuous functions on the sets $B_k$ 
and 
$$
\cp_{p}\left(\Omega\setminus\bigcup_{k=1}^{\infty}B_k\right)=0.
$$
Let  A be a compact set such that 
$A\subset\bigcup_{k=1}^{\infty}Int(B_k).$ 
Then for every 
$\varepsilon>0$ 
there exists a Borel set 
$A_{\varepsilon}\subset A$ 
such that 
$f\vert_{A_\varepsilon}$ 
is Lipschitz continuous and 
$\cp_{p}(A\setminus A_{\varepsilon})\leq \varepsilon.$ 
\end{thm}

\begin{proof}
By assumption there exists a monotone increasing sequence of closed sets 
$B_k$ 
such that 
$f\vert_{B_k}$ 
is Lipschitz continuous for any integer 
$k\geq 1$ 
and 
$\cp_{p}\left(\Omega\setminus \bigcup_{k=1}^\infty B_k\right)=0.$
Thus,
\begin{equation}
0=\cp_{p}\left(A\setminus \bigcup_{k=1}^\infty Int(B_k)\right)=\lim_{l\to \infty}\cp_{p}\left(A\setminus \bigcup_{k=1}^l Int(B_k)\right).
\end{equation}
Fix 
$\varepsilon>0.$ 
There exists a big enough integer $l_0$ such that
\begin{equation}
\cp_{p}\left(A\setminus \bigcup_{k=1}^l Int(B_k)\right)\leq\varepsilon,\quad \forall l\geq l_0. 
\end{equation}
Set $A_{\varepsilon}:=\bigcup_{k=1}^{l_0} Int(B_k)\cap A.$ It follows that $A_{\varepsilon}\subset A$ is a Borel set such that $f\vert_{A_{\varepsilon}}$ is Lipschitz continuous and $\cp_{p}\left(A\setminus A_{\varepsilon}\right)\leq \varepsilon.$      
\end{proof}
The following Corollary is an immediate consequence of Theorem 
$\ref{lem:LipCap}$ 
and Theorem 
$\ref{thm:Luzin's theorem for union of compact sets}.$  

\begin{cor}
\label{cor:Luzin property for capacity}
Let 
$\Omega\subset \R^n$ 
be an open set and 
$f\in W^2_{p,\loc}(\Omega),1\leq p<\infty.$
There exists a sequence of closed sets 
$\{B_k\}_{k=1}^{\infty}$, 
$B_k\subset B_{k+1}\subset\Omega$ 
for which the restrictions 
$f^* \vert_{B_k}$ 
are Lipschitz continuous
in 
$B_k$ up to a set of $p$-capacity zero
and 
$$
\cp_{p}\left(\Omega\setminus\bigcup_{k=1}^{\infty}B_k\right)=0.
$$
Moreover, if A is a compact set such that 
$A\subset\bigcup_{k=1}^{\infty}Int(B_k)$, 
then for every 
$\varepsilon>0$ 
there exists a Borel set 
$A_{\varepsilon}\subset A$ 
such that 
$f^*\vert_{A_\varepsilon}$ 
is Lipschitz continuous
in 
$A_\varepsilon$ up to a set of $p$-capacity zero
and 
$\cp_{p}(A\setminus A_{\varepsilon})\leq \varepsilon.$ 
\end{cor}

\section{\textbf{The change of variables formula}}

In this section we will derive from the results of the previous section the change of variables formula for mappings of the class $W^{2}_{p,\loc}(\Omega;\mathbb R^n)$.

\begin{rem}
Because differential and geometric properties of mappings 
$\varphi:\Omega\to\mathbb{R}^{n}$ 
are defined by their coordinate functions, the previous results on Lipschitz approximations of Sobolev functions are valid for mappings 
$\varphi:\Omega\to \R^n$.
 More precisely, one can generalize the results of Section 2 from the case of functions 
$f:\Omega\to \R$ 
to the case of mappings 
$\varphi:\Omega\to \R^n$ 
using that a mapping 
$\varphi:\Omega\to \R^n$ 
is a Sobolev mapping if and only if its coordinate functions are Sobolev functions and it is a Lipschitz mapping if and only if its coordinate functions are Lipschitz functions.
\end{rem}

The following theorem refines the formula of change of variables in the Lebesgue integral in the terms of the non-linear potential theory. The change of variables formula for Lipschitz and Sobolev mappings can be found, for example, in \cite{F69, H93,VGR79}.

\begin{thm}
\label{change}
Let 
$\Omega\subset \R^n$ 
be an open set and let 
$\varphi : \Omega\to \mathbb R^n$ 
be a measurable mapping such that
there exists a collection of closed sets 
$\{A_k\}_{k=1}^{\infty}$, $A_k\subset A_{k+1}\subset \Omega$ 
for which the restrictions 
$\varphi \vert_{A_k}$ 
are Lipschitz continuous mappings on the sets 
$A_k$ 
and 
$$
\cp_{p}\left(\Omega\setminus\bigcup_{k=1}^{\infty}A_k\right)=0.
$$
Then there exists a Borel set 
$S\subset \Omega$, $\cp_{p}(S)=0$, 
such that  the mapping
$\varphi:\Omega\setminus S \to \mathbb R^n$ 
has the Luzin $N$-property and the change of variables formula
 
\begin{equation}
\label{chvf}
\int\limits_A u\circ\varphi (x)|J(x,\varphi)|~dx=\int\limits_{\mathbb R^n\setminus \varphi(S)} u(y)N_\varphi(A,y)~dy,
\end{equation}
where 
$N_\varphi(A,y)$ 
is the multiplicity function defined as the number of preimages of $y$ in $A$ under 
$\varphi$, 
holds for every  measurable set 
$A\subset \Omega$ 
and every nonnegative Lebesgue measurable function $u: \mathbb R^n\to\mathbb R$.
\end{thm}

\begin{proof}
Denote $S:=\Omega\setminus\bigcup_{k=1}^{\infty}A_k.$ By assumption 
$\cp_{p}(S)=0.$ If $N\subset\Omega,|N|=0$, then
$\varphi(N\setminus S)=\bigcup_{k=1}^\infty
\varphi(A_k\cap N)$ 
and 
$\left|\bigcup_{k=1}^\infty
\varphi(A_k\cap N)\right|=0$ 
since 
$\varphi$ 
is Lipschitz continuous on 
$A_k$.

For each 
$k\in \N$
let us extend the Lipschitz continuous mapping 
$\varphi|_{A_k}$
to a Lipschitz continuous mapping defined on
$\R^n.$ 
By the change of variables formula for Lipschitz continuous mappings we have
\begin{equation}
\int\limits_{A\cap A_k} u\circ\varphi (x) |J(x,\varphi)|~dx=\int\limits_{\mathbb \varphi(A\cap A_k)} u(y)N_\varphi(A\cap A_k,y)~dy,
\end{equation}
for every measurable set 
$A\subset \Omega$ 
and every nonnegative measurable function 
$u: \mathbb R^n\to\mathbb R.$
The following two nonnegative sequences of functions are monotone increasing a.e as $k\to \infty$:
\begin{equation}
\chi_{A\cap A_k}u\circ\varphi |J(\cdot,\varphi)|\nearrow \chi_{A\setminus S}u\circ\varphi  |J(\cdot,\varphi)|,
\end{equation}
\begin{equation}
\chi_{\varphi(A\cap A_k)}uN_\varphi(A\cap A_k,\cdot)\nearrow \chi_{\varphi(A\setminus S)}uN_\varphi(A\setminus S,\cdot).
\end{equation}

Thus, by the Monotone Convergence Theorem we get

\begin{equation}
\label{eq:change of variable formula}
\int\limits_{A\setminus S} u\circ\varphi (x) |J(x,\varphi)|~dx=\int\limits_{\mathbb \varphi(A\setminus S)} u(y)N_\varphi(A\setminus S,y)~dy.
\end{equation}
Since 
$\cp_{p}(S)=0$ 
we have 
$|S|=0$, 
and one can replace 
$A\setminus S$ 
on the left hand side of 
$\eqref{eq:change of variable formula}$ 
by 
$A.$ 
Since the function 
$N_\varphi(A\setminus S,\cdot)$ 
vanishes outside 
$\varphi(A)$,
then one can replace 
$\varphi(A\setminus S)$ 
by 
$\R^n\setminus \varphi(S).$ 
It completes the proof.   
\end{proof}

By Theorem 
$\ref{lem:LipCap}$ 
Sobolev mappings of the class 
$W^2_{p,\loc}(\Omega;\R^n)$ 
satisfy the conditions of the change of variables formula and so for Sobolev mappings $\varphi:\Omega\to\mathbb R^n$ of the class $W^2_{p,\loc}(\Omega;\R^n)$ the change of variables formula \eqref{chvf} holds.

In the case of Sobolev mappings of the class $W^2_{1,\loc}(\Omega;\R^n)$,  
by Theorem
$\ref{lem:LipCap}$
with $p=1$,
Theorem 
$\ref{change}$
and 
$\eqref{eq:mutual absolute continuity of cap_1 and Hausdorff measure}$ 
we obtain the following corollary:

\begin{cor}
\label{changeH}
Let 
$\Omega\subset \R^n$
be an open set and 
$\varphi\in W^2_{1,\loc}(\Omega;\R^n)$.
Then there exists a Borel set 
$S\subset \Omega$, 
$\Haus^{n-1}(S)=0$, 
such that  the mapping
$\varphi:\Omega\setminus S \to \mathbb R^n$ 
has the Luzin $N$-property and the change of variables formula
 \begin{equation}
\label{chvfH}
\int\limits_A u\circ\varphi (x)|J(x,\varphi)|~dx=\int\limits_{\mathbb R^n\setminus \varphi(S)} u(y)N_\varphi(A,y)~dy,
\end{equation} 
holds for every  measurable set 
$A\subset \Omega$ 
and every nonnegative Lebesgue measurable function $u: \mathbb R^n\to\mathbb R$. 
\end{cor}

If the mapping 
$\varphi$ 
possesses the Luzin capacity-measure $N$-property (the image of a set of capacity zero has Lebesgue measure zero), then 
$|\varphi (S)|=0$ 
and the integral on the right hand side of 
$\eqref{chvf}$ 
can be rewritten as an integral on 
$\mathbb R^n$. We summarize it in the following corollary:

\begin{cor}
Let 
$\Omega\subset \R^n$ 
be an open set and let 
$\varphi\in W^2_{p,\loc}(\Omega;\R^n)$ 
which has the Luzin capacity-measure $N$-property. Then  the change of variables formula
\begin{equation}
\label{chvf1}
\int\limits_A u\circ\varphi (x)|J(x,\varphi)|~dx=\int\limits_{\mathbb R^n} u(y)N_\varphi(A,y)~dy,
\end{equation}
holds for every  measurable set 
$A\subset \Omega$ 
and every nonnegative Lebesgue measurable function $u: \mathbb R^n\to\mathbb R$.
\end{cor}

\vskip 0.3cm

{\bf Acknowledgments}:

The authors thank Vladimir Gol'dshtein for his useful remarks.

\vskip 0.3cm

Paz Hashash; Department of Mathematics, Ben-Gurion University of the Negev, P.O.Box 653, Beer Sheva, 8410501, Israel 
 
\emph{E-mail address:} \email{pazhash@post.bgu.ac.il} \\

Alexander Ukhlov; Department of Mathematics, Ben-Gurion University of the Negev, P.O.Box 653, Beer Sheva, 8410501, Israel 
							
\emph{E-mail address:} \email{ukhlov@math.bgu.ac.il

\end{document}